\documentclass[a4paper,11pt,reqno]{amsart}

\usepackage{amsthm, mathrsfs,amssymb,amsmath,}
\usepackage{enumerate}
\usepackage{mathtools}
\makeatletter
\@namedef{subjclassname@2010}{%
	\textup{2010} Mathematics Subject Classification}
\makeatother

\usepackage{hyperref}
\allowdisplaybreaks
\numberwithin{equation}{section}

\usepackage{graphicx}
\usepackage{enumitem}

\newtheorem{theorem}{Theorem}[section]
\newtheorem{lemma}[theorem]{Lemma}

\theoremstyle{definition}
\newtheorem{definition}[theorem]{Definition}
\newtheorem{example}[theorem]{Example}

\newtheorem{proposition}[theorem]{Proposition}
\theoremstyle{remark}
\newtheorem{remark}[theorem]{Remark}

\numberwithin{equation}{section}

\usepackage{fancyhdr}
\begin{document}
	
	\title[Fractal dimension and Fractional calculus]{Vector-valued fractal functions: Fractal dimension and Fractional calculus}
	

	
	\author{Manuj Verma}
	\address{Department of Mathematics, IIT Delhi, New Delhi, India 110016}
	\email{mathmanuj@gmail.com}
	\author{Amit Priyadarshi}
	\address{Department of Mathematics, IIT Delhi, New Delhi, India 110016}
	
	\email{priyadarshi@maths.iitd.ac.in}
	\author{Saurabh Verma}
	\address{Department of Applied Sciences, IIIT Allahabad, Prayagraj, India 211015 }
	\email{saurabhverma@iiita.ac.in}
	
	
	\date{\today}

	\begin{abstract}
		There are many research available on the study of real-valued fractal interpolation function  and fractal dimension of its graph. In this paper, our main focus is to study the dimensional results for  vector-valued fractal interpolation function and its Riemann-Liouville fractional integral. Here, we give some results which ensure that dimensional results for vector-valued functions are quite different from real-valued functions. We determine interesting bounds for the Hausdorff dimension of the graph of vector-valued fractal interpolation function. We also obtain bounds for the Hausdorff dimension of associated invariant measure supported on the graph of vector-valued fractal interpolation function. Next, we  discuss more efficient upper bound for the Hausdorff dimension of measure in terms of probability vector and contraction ratios. Furthermore, we determine some dimensional results for the graph of the Riemann-Liouville fractional integral of a vector-valued fractal interpolation function.
	\end{abstract}
	
	\maketitle

	
	
	\section{INTRODUCTION}

	In Fractal Geometry, calculation of fractal dimension is one of the major themes for researchers. There are a lot of research available on the study of fractal dimension of sets and the graph of continuous functions, see for more details \cite{MF2,Fal,SP,w&l}.
	In 1986, Barnsley \cite{MF1} established  the theory of real-valued fractal interpolation functions (FIFs) by using the concept of iterated function systems (IFSs) and estimated the Hausdorff dimension of affine FIFs. Also, there are some articles \cite{MF6,B&H}, in which researchers calculated the box dimension of affine FIFs.  In 1991, Massopust \cite{MP} introduced vector-valued fractal interpolation functions and computed the box dimension of the graph of vector-valued fractal interpolation functions. After that,  Hardin and Massopust \cite{HM} constructed fractal interpolation functions from $\mathbb{R}^n$ to $ \mathbb{R}^m$ and derived a formula for the box dimension of hidden variable FIFs. Recently,  Barnsley and Massopust \cite{MF3} introduced the construction of bi-linear FIFs and gave a formula for the box dimension of the graph of bi-linear FIFs. Also, there are some research available on the study of fractal surfaces, see \cite{M2,SS2} for more details.
	\par Fractional calculus is a very old area of research and it attracts researchers for its wide applications. There are many researchers, who investigated relationship between fractals and fractional calculus, see for more details \cite{Tat, R.R}. In 2009, Ruan et al. estimated the box dimension of the graph of Riemann-Liouville fractional integral of linear FIFs. Gowrishankar and Uthayakumar \cite{Gowri} also studied theory regarding  Riemann-Liouville fractional integral of countable linear FIFs. In 2010, Liang \cite{Liang1} showed that the box dimension of the  Riemann-Liouville fractional integral of a continuous function of bounded variation is 1. After that, in 2018, Liang \cite{Liang4} proved that the upper box dimension of the Riemann-Liouville fractional integral of a continuous function can not exceed the upper box dimension of this continuous function. Some authors \cite{Liang5,Liang EST,Z} estimated fractal dimension of the graph of the Riemann-Liouville fractional integral of functions like H\"older function and unbounded variation function. Abbas and Chandra \cite{SS2,SS} worked on mixed  Riemann-Liouville fractional integral of fractal surfaces. In \cite{MR1}, Roychowdhury and Selmi determined quantization dimension for the invariant measure of the recurrent IFS. Some recent works on dimension of self-similar measures can be seen in \cite{MHOCHMAN,PABLO} and the references cited therein.  However, we should emphasize on the fact that the present paper deals with the measures supported on the graphs of fractal functions, these measures do not belong to the class of self-similar (or self-affine) measures. 
	Here, firstly we convince the reader that dimensional results for vector-valued function and real-valued function are not the same, in general. In this paper, our main focus is to describe some dimensional results for the graph of vector-valued fractal interpolation functions and its Riemann-Liouville fractional integral.  

	\par
	The paper is organized as follows. In the forthcoming Section \ref{se2}, we give some definitions and required results for our work. Section \ref{se3} consists of some dimensional results on the vector-valued fractal function and associated measures. In this section, we first give some propositions and lemmas which describe relations between fractal dimension of vector-valued function and its components. After that, we establish bounds for the  Hausdorff dimension of vector-valued fractal interpolation function. We also obtain bounds for the Hausdorff dimension of associated invariant measures supported on the graph of vector-valued fractal interpolation function. Then, we determine more general upper bound for the Hausdorff dimension of invariant measure in terms of probability vector and contraction ratios. Later, we give some conditions under which vector-valued fractal interpolation function belongs to some special class of function spaces like H\"older space, bounded variation class and absolutely continuous function space. In Section \ref{sc-4}, we prove that the Riemann-Liouville fractional integral of a vector-valued fractal interpolation function is again a vector-valued fractal interpolation function corresponding to some data set. Next, we compute  the Hausdorff dimension and box-counting dimension of the graph of the Riemann-Liouville fractional integral of a vector-valued fractal interpolation function.
	
	\section{preliminaries}\label{se2}
	\begin{definition}
		Let $F$ be a subset of a metric space $(X,d)$. The Hausdorff dimension of $F$ is defined as follows
		$$ \dim_H{F}=\inf\{\beta>0: \text{for every}~\epsilon>0,~\text{there is a cover}~~ \{U_i\}_{i=1}^{\infty}~\text{of}~F~\text{ with}\sum_{i=1}^{\infty} |U_i|^\beta<\epsilon \},$$
		where $|U_i|$ denotes the diameter of $U_i.$  
		
	\end{definition}
	\begin{definition}
		The box dimension of a non-empty bounded subset $F$ of $(X,d)$ is defined as
		$$\dim_{B}F=\lim_{\delta \to 0}\frac{\log{N_{\delta}(F)}}{-\log\delta},$$
		where $N_{\delta}(F)$ denotes the smallest number of sets of diameter at most $\delta$ that can cover $F,$ provided the limit exists. 
		If this limit does not exist then the upper and the lower box dimension, respectively, are defined as 
		$$\overline{\dim}_{B}F=\limsup_{\delta \to 0}\frac{\log{N_{\delta}(F)}}{-\log\delta},$$
		$$\underline{\dim}_{B}F=\liminf_{\delta \to 0}\frac{\log{N_{\delta}(F)}}{-\log\delta}.$$
	\end{definition}
	\begin{definition} Let $F$ be a subset of a metric space $(X,d).$
		For $r> 0$ and $t\geq0$, let
		$P^t_r(F)=\sup\big \{\sum^{}_i|B_i|^t  \big\}$, where $\{B_i\}$ is a collection of disjoint balls of radii at most $r$ with centres in $F$.
		As $r$ decreases, $P^t_r(F)$ also decreases. Therefore, the limit
		$$P^t_0(F)= \lim_{r \to 0}P^t_r(F)$$
		exists. We define
		$$P^t(F)=\inf\bigg\{\sum_iP^t_0(F_i): F\subset\bigcup^{\infty}_{i=1}F_i\bigg\},$$
		and it is known as the $t$-dimensional packing measure. The Packing dimension is defined as follows:
		$$\dim_{P}(F)=\inf\{t\geq 0 : P^{t}(F)=0\}=\sup\{t\geq 0 : P^{t}(F)=\infty\}.$$
	\end{definition}
	\textbf{Note-} We denote the graph of a function $f$ by $G(f)$ throughout this paper.
	\begin{definition}
		Let $\mu$ be a Borel probability measure on $\mathbb{R}^n$, where $n\in\mathbb{N}.$ The Hausdorff dimension of measure $\mu$ is defined by 
		$$\dim_H(\mu)= \inf\{\dim_H(F): F ~\text{is a Borel subset such that } \mu(F)>0\}.$$
	\end{definition}
	\begin{definition}
		The lower and upper local dimensions of a measure $\mu$ at $x\in \mathbb{R}^n$, respectively, are given by
		$$\underline{\dim}_{loc}\mu(x)=\liminf_{r\to 0}\frac{\log\mu(B(x,r))}{\log r},~~~\overline{\dim}_{loc}\mu(x)=\limsup_{r\to 0}\frac{\log\mu(B(x,r))}{\log r}. $$
	\end{definition}
	\begin{proposition}\cite{Fal1}\label{loc}
		For a finite Borel measure $\mu$, we have
		$$\dim_H(\mu)\leq \inf\{s: \underline{\dim}_{loc}\mu(x)\leq s~~ \text{for}~~\mu\text{-almost all}~~x\in \mathbb{R}^n \}.$$
	\end{proposition}
	\begin{theorem}\cite{Fal}
		Let $f: [0,1] \to \mathbb{R}$ be a H\"{o}lder continuous function with the H\"{o}lder exponent $ \sigma\in (0,1)$. Then   $\overline{\dim}_B(G(f))\leq 2-\sigma. $ 
		
	\end{theorem}
	
	
	
	\subsection{Iterated Function Systems} Let $(X, d)$ be a complete metric space, and we denote the family of all nonempty compact subsets of $X$ by $H(X)$ . We define the Hausdorff metric
	$$\mathcal{D}(A,B) = \inf\{\delta>0 : A\subset B_\delta~~\text{and}~~ B \subset A_\delta \} ,$$
	where $A_\delta$ and $B_\delta$ denote the $\delta $-neighbourhoods of sets $A$ and $B$, respectively.
	Then it is well-known that $(H(X),\mathcal{D})$ is a complete metric space. A map $f: (X,d) \to (X,d) $  is called a contraction if there exists a constant  $r<1$ such that 
	$$d(f(x),f(y)) \le r d(x,y),~~\forall~~~ x , y \in X.$$
	\begin{definition}
		The system $\mathcal{I}=\big\{(X,d); f_1,f_2,\dots,f_N \big\}$ is called an iterated function system (IFS), if each $f_i$ is a contraction self-map on $X$ for $i\in \{1,2,\dots,N\}$.
	\end{definition}
	Let $\mathcal{I}=\big\{(X,d); f_1,f_2,\dots,f_N \big\}$ be an IFS. We define a mapping (widely known as the Hutchinson operator) $S$ from $H(X)$ into $H(X)$ given by
	$$ S(A) = \bigcup\limits_{i=1}^N f_i (A).$$
	The map $S$ defined above is a contraction
	map under the Hausdorff metric $\mathcal{D}$. If $(X,d)$ is a complete metric space, then, by the Banach contraction principle, there exists a unique $F\in H(X)$ such that $ F = \cup_{i=1}^N f_i (F)$, and it is called the attractor of the IFS. We refer the reader to see \cite{MF2,Fal}
	for details.
	\begin{definition}
		We say that an IFS $\mathcal{I}=\{(X,d);f_1,f_2,\dots,f_N\}$ satisfies the open set condition (OSC) if there is a non-empty open set $U$ with $f_i(U) \subset U~~\forall~i\in \{1,2,\cdots,N\}$  and  $ f_i(U)\cap f_j(U)= \emptyset $ for $i\ne j$. Moreover, if $U \cap F \ne \emptyset,$  where $F$ is the attractor of the IFS $\mathcal{I}$, then we say that  $\mathcal{I}$ satisfies the strong open set condition (SOSC). If $f_i(F)\cap f_j(F)=\emptyset$ for $i\ne j$, then we say that the IFS $\mathcal{I}$ satisfies the  strong separation condintion (SSC).  
	\end{definition}
	\subsection{Fractal Interpolation Functions}
	Consider a set of data points $ \{(x_i,\boldsymbol{y}_i)\in \mathbb{R}\times\mathbb{R}^M : i=1,2,\dots,N\} $ with $x_1<x_2<\dots <x_N$. Set $ T = \{1,2,...,N-1\}$ and   $J= [x_1, x_N] .$ For each $ k \in T,$ set $J_k= [x_k, x_{k+1}]$ and let $P_k: J \rightarrow J_k $ be a contractive homeomorphism  satisfying
	$$ P_k(x_1)=x_k,~~P_k(x_N)=x_{k+1}.$$
	For each $ k\in T $, let $F_k: J\times \mathbb{R}^M \rightarrow \mathbb{R}^M $ be a continuous  map such that 
	$$ |F_k(t,\boldsymbol{z_1}) - F_k(t,\boldsymbol{z_2})| \leq \tau_k |\boldsymbol{z_1}- \boldsymbol{z_2}| ,$$
	$$ F_k(x_1,\boldsymbol{y}_1)=\boldsymbol{y}_k, F_k(x_N,\boldsymbol{y}_N)=\boldsymbol{y}_{k+1},$$
	where $(t,\boldsymbol{z_1}), (t,\boldsymbol{z_2}) \in J\times \mathbb{R}^M $ and $ 0 \leq \tau_k < 1.$  In particular, we can take for each $k\in T$,
	$$P_k(t)=a_k t+ d_k, \quad F_k(t,\boldsymbol{z}) = \alpha_k \boldsymbol{z} + q_k (t).$$
	In the above expressions $a_k$ and $d_k$ are uniquely determined by the condition $ P_k(x_1)=x_k, P_k(x_N)=x_{k+1}.$ The multiplier $\alpha_k$ is called the scaling factor, which satisfies $-1< \alpha_k <1$ and $q_k:J\rightarrow \mathbb{R}^M$ is a continuous function such that $q_k(x_1)=\boldsymbol{y}_k-\alpha_k \boldsymbol{y}_1$ and $q_k(x_N)=\boldsymbol{y}_{k+1}-\alpha_k \boldsymbol{y}_N$.
	Now for each $k \in T $, we define function $W_k:J\times \mathbb{R}^M \rightarrow  J\times \mathbb{R}^M $ by  $$W_k(t,\boldsymbol{z})=\big(P_k(t),F_k(t,\boldsymbol{z})\big). $$
	Then the IFS $\mathcal{I}:=\{J\times \mathbb{R}^M ;W_1,W_2,\dots,W_{N-1}\}$ has a unique attractor \cite[Theorem 1]{MF1}, which is the graph of a function $h$ which satisfies
	the following functional equation reflecting self-referentiality:
	$$h(t)= \alpha_k h \big(P_k^{-1}(t) \big)+ q_k \big(P_k^{-1}(t)\big), t \in J_k, k \in T.$$
	The above function $h$ is known as the fractal interpolation function (FIF).
	\par   
	Let $\mathcal{I}:=\{J\times \mathbb{R}^M ;W_1,W_2,\dots,W_{N-1}\}$ be an iterated function system (IFS) with probability vector $(p_1,p_2,\cdots,p_{N-1}).$ Let $T=\{1,2,\cdots,N-1\}.$ Let $T^n$ be the set of all finite sequences of $T$ of length $n$,
	that is, $T^n=\{(i_1,i_2,\cdots,i_n): 1\leq i_j\leq N-1\}$ and let $T^*=\cup_{n\in \mathbb{N}}T^n$ denote the set of all sequences of $T$ of finite length. We denote by $\Omega$ the set of all infinite sequences   $\Omega=\{(\omega_1,\omega_2,\cdots): 1\leq \omega_j\leq N-1\}$. If $\bold{i}=(i_1,i_2,\cdots,i_n)\in T^*$, then the set $\{\omega\in \Omega: \omega_j=i_j~~ \forall~~ j\in\{1,2,\cdots,n\} \}$ is called a cylinder of length $n$ in $\Omega$ generated by $\bold{i}$  and is denoted by $[\bold{i}].$ For $\bold{i}=(i_1,i_2,\cdots,i_n)\in T^*$, we write $p_\bold{i}=p_{i_1}p_{i_2}\cdots p_{i_n}.$ Let $\mathcal{B}$ denote the Borel sigma-algebra on $\Omega$ generated by cylinders in $\Omega.$ We define a Borel probability measure $\mu$ on $\Omega$ by first defining on cylinders in $\Omega$ as follows:
	$$\mu([\bold{i}])=p_{\bold{i}}~~\forall~~ \bold{i}=(i_1,i_2,\cdots,i_n)\in T^* .$$
	Then, by the Caratheodory's extension theorem, $\mu$ can be extended to a unique Borel probability measure on $\Omega$ which we again denote by $\mu.$ It is clear that the support of $\mu$ is $\Omega.$

	\section{Dimensional results on  the vector-valued fractal function and associated measures}\label{se3}
	In the following lemma, we provide a relationship between the Hausdorff dimension of vector-valued continuous function and the Hausdorff dimension of its components. 
	\begin{lemma}\label{new222}
		Let $f: [a,b] \rightarrow \mathbb{R}^M$ be a continuous function and $f_i:[a,b] \to \mathbb{R}$ be the $i$th component of $f$, that is,  $f=(f_1,f_2,\cdots,f_M) $. Then we have
		\begin{itemize}
			\item[(1)] \label{t1} $\dim_H (\text{G}(f))\geq \max\limits_{1\leq i\leq M}\{\dim_H (\text{G}(f_i)) \}.$ 
			\item[(2)] $\dim_H (\text{G}(f)) = \dim_H( \text{G}(f_i))$,  provided the component function  $f_j$ is Lipschitz for each $1\leq j\ne i\leq M.$
		\end{itemize}
		
	\end{lemma}
	\begin{proof}
		\begin{itemize}
			\item[(1)]  Let us define a mapping $\Phi : G(f) \to G(f_i)$ as follows
			$$\Phi(x,f(x) )=(x,f_i(x)).$$
			One can easily prove that $\Phi$ is a Lipschitz map.
			Now, the Lipschitz invariance property of the Hausdorff dimension (Cf. \cite[Corollary 2.4(a)]{Fal}) yields $$\dim_H (\text{G}(f))\geq \dim_H (\text{G}(f_i)).$$
			Since the above inequality holds for any $i\in \{1,2,\cdots,M\}$, we get 
			$$\dim_H (\text{G}(f))\geq \max\limits_{1\leq i\leq M}\{\dim_H (\text{G}(f_i)) \},$$
			completing the proof of item (1).
			\item[(2)]  In this part, we continue our proof with the same mapping $\Phi : G(f) \to G(f_i),$ defined by 
			$$\Phi(x,f(x) )=(x,f_i(x)).$$ 
			By using $f_j$ is Lipschitz for each $1\leq j\ne i\leq M,$ one can easily show that the map $\Phi$ is a  bi-Lipschitz map.
			In the light of the bi-Lipschitz invariance property of the Hausdorff dimension (Cf. \cite[Corollary 2.4(b)]{Fal}), we get
			$$\dim_H(\text{G}(f))= \dim_H (\text{G}(f_i)),$$
			which completes the proof.
			
		\end{itemize}
		
	\end{proof}
	
	Now, we describe some results similar to the above in terms of other dimensions.
	\begin{proposition}
		Let $f: [a,b] \rightarrow \mathbb{R}^M$ be a continuous function and $f_i:[a,b] \to \mathbb{R}$ be the $i$th component of $f$, that is, $f=(f_1,f_2,\cdots,f_M) $. Then we have
		$$\dim_P (\text{G}(f))\geq \max\limits_{1\leq i\leq M}\{\dim_P (\text{G}(f_i)) \},$$     $$\overline{\dim}_B (\text{G}(f))\geq \max\limits_{1\leq i\leq M}\{\overline{\dim}_B (\text{G}(f_i)) \},$$                         $$\underline{\dim}_B (\text{G}(f))\geq \max\limits_{1\leq i\leq M}\{\underline{\dim}_B (\text{G}(f_i)) \}.$$
	\end{proposition}
	\begin{proof}
		The proof is similar to part (1) of Lemma \ref{t1}. Hence, we omit it.
	\end{proof}
	
	\begin{proposition}\label{prop3.3} Let $f: [a,b] \rightarrow \mathbb{R}^M$ be a continuous function and $f_i:[a,b] \to \mathbb{R}$ be the $i$th component of $f$, that is, $f=(f_1,f_2,\cdots,f_M) $. If $f_j$ is a Lipschitz function for each $1\leq j\ne i\leq M$, then we have
		$$\dim_P(\text{G}(f))= \dim_P (\text{G}(f_i)),~~\overline{\dim}_B(\text{G}(f))= \overline{\dim}_B (\text{G}(f_i)),$$ and 
		$$\underline{\dim}_B(\text{G}(f))= \underline{\dim}_B (\text{G}(f_i)).$$
	\end{proposition}
	
	\begin{proof}
		The proof is similar to part (2) of Lemma \ref{t1}, hence omitted.  
	\end{proof}
	\begin{lemma}
		Let $f: [a,b] \rightarrow \mathbb{R}^M$ be a continuous function and $f_i:[a,b] \to \mathbb{R}$ be the $i$th component of $f$, that is, $f=(f_1,f_2,\cdots,f_M) .$ If $f_j$ is a Lipschitz function on $[a,b]$ for all $1\leq j\ne i\leq M$, then
		$$\dim_H (\text{G}(f)) =\dim_H \bigg(\text{G}\bigg(\sum\limits_{j=1}^{M}f_j\bigg)\bigg)= \dim_H( \text{G}(f_i)),$$
		$$\overline{\dim}_B (\text{G}(f)) =\overline{\dim}_B \bigg(\text{G}\bigg(\sum\limits_{j=1}^{M}f_j\bigg)\bigg)= \overline{\dim}_B( \text{G}(f_i)),$$
		$$\underline{\dim}_B (\text{G}(f)) =\underline{\dim}_B \bigg(\text{G}\bigg(\sum\limits_{j=1}^{M}f_j\bigg)\bigg)= \underline{\dim}_B( \text{G}(f_i)),$$
		$$\dim_P (\text{G}(f)) =\dim_P \bigg(\text{G}\bigg(\sum\limits_{j=1}^{M}f_j\bigg)\bigg)= \dim_P( \text{G}(f_i)).$$
	\end{lemma}
	\begin{proof}
		We define a mapping $\Phi :\text{G}\bigg(\sum\limits_{j=1}^{M}f_j\bigg)\to G(f_i) $ as follows 
		$$\Phi\bigg(x,\bigg(\sum\limits_{j=1}^{M}f_j(x)\bigg)\bigg)=(x,f_i(x)).$$
		By using simple property of norm and using the condition that the function $f_j$ is  Lipschitz on $[a,b]$ for all $1\leq j\ne i\leq M$, it is easy to show that $\Phi$ is a bi-Lipschitz map.
		In the light of Lemma \ref{new222}, Proposition \ref{prop3.3} and by using the bi-Lipschitz invariance property of Hausdorff dimension, lower box dimension, upper box dimension, and Packing dimension, we get our required result.
	\end{proof} 
	\begin{remark}\label{new9987}
		Note that in \cite{Kono} the Peano space filling  curve ${\Theta}:[0,1] \to [0,1] \times[0,1]$ is a H\"older continuous function with exponent $\frac{1}{2}$. Let $\Theta_i$ be the $i$th component of $\Theta$ for $i\in \{1,2\}$. The component function $\Theta_i$ satisfy $\dim_H\big(G({\Theta_i})\big) = 1.5 $ for all $i\in \{1,2\}$. However, we have ${\dim}_H\big({G}(\Theta)\big) \geq 2.$ From this, it is clear that the upper bound of the Hausdorff dimension of the graph of a vector-valued function cannot be written in terms of its H\"{o}lder exponent as we do for a real-valued function. 
		\par
		From the above discussion, we conclude that the dimensional results for vector-valued functions can be different from that of real-valued functions.
	\end{remark}
	In the rest of the section, we follow the notations introduced in Section \ref{se2}. Motivated by the work of Barnsley and Massopust \cite{MF3} on real-valued fractal functions, we define a metric $d_*$ on the space $ J\times \mathbb{R}^M$ as follows
	$$d_*((t_1,\boldsymbol{z_1}),(t_2,\boldsymbol{z_2}))= |t_1-t_2|+\|(\boldsymbol{z_1}-h(t_1))-(\boldsymbol{z_2}-h(t_2))\|~~~~~~~\forall~~(t_1,\boldsymbol{z_1}),(t_2,\boldsymbol{z_2})\in J\times \mathbb{R}^M.$$
	Then, it is easy to check that $\big( J\times \mathbb{R}^M,d_* \big)$ is a complete metric space. Further, it can also be seen that the above metric $d_*$ is equivalent to the Euclidean metric on $ J\times \mathbb{R}^M.$
	\par
	In the next theorem, we prove that $W_k$ is a contraction map for each $k\in T.$
	\begin{theorem}\label{W_K}
		Let $W_k(t,\boldsymbol{z})=(P_k(t),F_k(t,\boldsymbol{z}))$, where $P_k$ and $F_k$ are defined as in Section \ref{se2}. Then the map $W_k: J\times \mathbb{R}^M \to J\times \mathbb{R}^M $ is a contraction map on the complete metric space $\big( J\times \mathbb{R}^M,d_* \big).$
	\end{theorem}
	\begin{proof}
		Let $(t_1,\boldsymbol{z_1}),(t_2,\boldsymbol{z_2})\in J\times \mathbb{R}^M.$ Then we have 
		\begin{align*}
			d_*(W_k(t_1,\boldsymbol{z_1}),W_k(t_2,\boldsymbol{z_2}))&=d_*((P_k(t_1),F_k(t_1,\boldsymbol{z_1})),(P_k(t_2),F_k(t_2,\boldsymbol{z_2})))\\
			&=|P_k(t_1)-P_k(t_2)| + \|(F_k(t_1,\boldsymbol{z_1})-h(P_k(t_1)))\\&-(F_k(t_2,\boldsymbol{z_2})-h(P_k(t_2)))\|\\
			&=|P_k(t_1)-P_k(t_2)| + \|(\alpha_k \boldsymbol{z_1}+q_k(t_1)-h(P_k(t_1)))\\&-(\alpha_k \boldsymbol{z_2}+q_k(t_2)-h(P_k(t_2)))\|\\&
			\leq |a_k||t_1-t_2|+\|\alpha_k(\boldsymbol{z_1}-h(t_1))-\alpha_k(\boldsymbol{z_2}-h(t_2))\|\\&
			\leq\max\{|a_{k}|,|\alpha_{k}|\}~~ d_*((t_1,\boldsymbol{z_1}),(t_2,\boldsymbol{z_2})).
		\end{align*}
		Since $\max\{|a_{k}|,|\alpha_{k}|\} <1$, it follows that $W_k$ is a contraction map for each $k\in T.$
	\end{proof}
	\par
	In the following theorem, we describe bounds for the Hausdorff dimension of the graph of vector-valued FIF.
	
	\begin{theorem}\label{th3.6}
		Let $\mathcal{I}:=\{J\times \mathbb{R}^M;~~W_k :k\in T\}$ be the IFS as defined earlier such that $$  c_k d((t_1,\boldsymbol{z_1}),(t_2,\boldsymbol{z_2}) ) \le d(W_k(t_1,\boldsymbol{z_1}) , W_k(t_2,\boldsymbol{z_2})) \le C_k d((t_1,\boldsymbol{z_1}),(t_2,\boldsymbol{z_2})) ,$$ where  $(t_1,\boldsymbol{z_1}),(t_2,\boldsymbol{z_2}) \in J\times \mathbb{R}^M$ and $0 < c_k \le C_k < 1 ~ \forall~ k \in T .$ Then $r \le  \dim_H(G(h)) \le R  ,$ where $h$ is fractal function and $r$, $R$ are given by  $ \sum\limits_{k\in T} c_k^{r} =1$ and $ \sum\limits_{k\in T} C_k^{R} =1$, respectively.
	\end{theorem}
	
	\begin{proof}
		By using Proposition $9.6$ in \cite{Fal}, we easily get the required upper bound of the Hausdorff dimension of the graph of fractal function $h$. Next, we shall determine a lower bound of $\dim_H(G(h)).$
		Let $O_* = (x_1,x_N) \times \mathbb{R}^M.$ Then  $ W_i(O_*) \cap W_{j}(O_*)=\emptyset$ for each  $i\ne j \in T$ because $P_i\big((x_1,x_N)\big) \cap P_j\big((x_1,x_N)\big)=\emptyset   ~~~~\forall~~i\ne j \in T.$ We can easily observe that for each $i\in T$, $W_i(O_*)\subset O_*$ and  $ O_* \cap G(h) \ne \emptyset.$ This implies that the  IFS $\mathcal{I}$ satisfies the SOSC. Since  $ O_* \cap G(h) \ne \emptyset,$ there exists an index $\sigma\in T^*$ such that $W_\sigma(G(h))\subset O_*,$ where $T^*:=\cup_{n \in \mathbb{N}}T^n$. For each $n\in \mathbb{N}$, let $B_n^*$ be the attractor of the IFS  $\mathcal{I}_n=\{J\times \mathbb{R}^M;~~W_{i\sigma}: i \in T^n\}$. After observing the code space of the IFS $\mathcal{I}$ and the IFS $\mathcal{I}_n$, it is clear that $B_n^*\subseteq G(h)$. For $i\ne j \in T^n$, $W_i(O_*) \cap W_{j}(O_*)=\emptyset.$ It follows that $W_{i\sigma}(B_n^*)\cap W_{j\sigma}(B_n^*)=\emptyset$  for $i\ne j\in T^n$. Thus, the IFS $\mathcal{I}_n=\{J\times \mathbb{R}^M; ~W_{i\sigma}: i \in T^n\}$ satisfies the SSC. So, it is clear that the IFS $\mathcal{I}_n$ satisfies the hypothesis of Proposition $9.7$ in \cite{Fal}. Therefore, Proposition $9.7$ in \cite{Fal} yields that $ r_n \le \dim_H(B_n^*)$, where $r_n$ is given by  $ \sum_{ i \in T^n} c_{i\sigma}^{r_n} =1.$ Thus,  $  r_n \le \dim_H(B_n^*) \le \dim_H(G(h))$ because $B_n^*\subset G(h)$. Suppose that $ \dim_H(G(h)) < r.$ This implies that  $ r_n < r $ for all $n\in \mathbb{N}$. Let $ c_{max}=\max\{c_1, c_2, \dots,c_{N-1}\}.$ We have
		$$
		c_{\sigma}^{- r_n}  = \sum_{ i \in T^n} c_{i}^{r_n}\ \ge \sum_{ i \in T^n} c_{i}^{r} c_{i}^{\dim_H(G(h)) -r} \ge \sum_{ i \in T^n} c_{i}^{r} c_{max}^{n(\dim_H(G(h)) - r)}. $$
		This implies that $$c_{\sigma}^{- r} \geq c_{max}^{n(\dim_H(G(h)) -r)}. $$ 
		We get a contradiction for large values of $n\in \mathbb{N} $. Thus, our assumption is wrong. Therefore, we obtain $ \dim_H(G(h)) \ge  r.$ Thus, the proof is complete.
	\end{proof}
	\begin{remark}
		In the above theorem, we can also take the mapping $P_k$ to be non-linear for all $k\in T$ and also it is not necessary that $\alpha_k$'s are constants. That is, one can consider scalings function $\alpha_k: J \to \mathbb{R}$ such that $\|\alpha_k\|_{\infty} < 1$ ( see, for instance, \cite{M2}).
	\end{remark}
	\begin{remark}
		We may compare the above result with Theorem $4$ of \cite{MF1}, wherein Barnsley proved the aforesaid theorem on dimension of FIF using potential theoretic approach. However, our approach is different from his approach, and our result also gives the same dimension bounds under less restrictive conditions. To be precise, following the same notation as in \cite{MF1}, the assumptions taken therein
		$$ t_1\cdot t_N \le  (\text{Min}\{a_1, a_N \}) \Big(\sum_{n=1}^N t_n^l\Big)^{2/l}~\text{and}~ L_n(x)=a_nx+h_n$$ are not needed to obtain the required lower bound. In this paper, we denote $L_n$ by $P_n$, $h_n $ by $d_n$ and  $t_n$ by $c_n.$
	\end{remark}
	\begin{remark}
		In \cite{MP}, Massopust calculated the exact value of the box dimension of the linear affine vector-valued FIF using covering method under some conditions. However, our above result gives bounds for the Hausdorff dimension of non-linear vector-valued FIF.  
	\end{remark} 
	\begin{remark}
		Barnsley and Massopust gave a formula for the exact value of the box dimension of bilinear FIFs using their own technique in Theorem $6$ of \cite{MF3}. However, it is worth to note that if we choose $q_k(x)= h(l_k(x))-S_k(l_k(x))b(x)$ and $\alpha_k(x)=S_k(l_k(x))$, then our class for which we give dimensional result will reduce to the class of bilinear FIFs. Thus, our result gives dimensional result for a much bigger class.
	\end{remark}
	\begin{remark}
		In \cite{MR}, Roychowdhury has considered hyperbolic recurrent IFS consisting of the bi-Lipschitz mapping. He has obtained bounds of the Hausdorff and box dimension of this IFS by using volume arguments and pressure function under the open set condition. Note that the recurrent IFS is a generalized version of the IFS, hence so is Roychowdhury's result. However, in our proof we do not use any pressure function and volume argument. Thus, our proof can be  done for general complete metric spaces. 
	\end{remark}
	The next theorem can be obtained by \cite{H}, however, we include its detailed proof for completeness and record.
	\begin{theorem}\label{th3.9}
		Let $W_k:J\times \mathbb{R}^M \rightarrow  J\times \mathbb{R}^M $ be defined by $W_k(t,\boldsymbol{z})=\big(P_k(t),F_k(t,\boldsymbol{z})\big) $ as earlier. Also, let $(p_1,\dots, p_{N-1})$ be a probability vector. Then there exists a unique Borel probability measure $\mu_*$ supported on the graph $G(h)$ of the fractal interpolation function such that $$\mu_*=\sum_{k \in T}p_k\mu_* \circ W_k^{-1}. $$
	\end{theorem}
	\begin{proof}
		In view of Theorem \ref{W_K}, we have a metric $d_*$ with respect to which each mapping $W_k:J\times \mathbb{R}^M \rightarrow  J\times \mathbb{R}^M $ will be contraction with contracting ratio $C_k=\max\{|a_{k}|,|\alpha_{k}|\}$ for each $k\in T$. 
		Recall that the collection of all Borel probability measures on $J \times \mathbb{R}^M$, denoted by $\mathcal{P}(J \times \mathbb{R}^{M})$, is a complete metric space with respect to the Hutchinson metric $d_H$ defined as $$d_H(\mu,\nu )=\sup \Big\{\Big|\int f d\mu(t,\boldsymbol{z}) - \int f d\nu(t,\boldsymbol{z})\Big|: ~f:J \times \mathbb{R}^M \to \mathbb{R}, \mbox{ Lip}(f) \le 1 \Big\},$$ where the supremum is taken over all Lipschitz functions $f: J\times \mathbb{R}^M \to \mathbb{R}$ satisfying $ \mbox{Lip}(f):= \inf\{L_f>0: |f(t_1,\boldsymbol{z_1})-f(t_2,\boldsymbol{z_2})| \le L_f d_*((t_1,\boldsymbol{z_1}),(t_2,\boldsymbol{z_2})) ~~~\forall~(t_1,\boldsymbol{z_2}),(t_2,\boldsymbol{z_2}) \in  J\times \mathbb{R}^M\}<1.$
		Define a mapping $\mathcal{M}: \mathcal{P}(J \times \mathbb{R}^{M}) \to \mathcal{P}(J \times \mathbb{R}^{M})$ by $\mathcal{M}(\mu)=\sum_{k \in T} p_k \mu  \circ W_k^{-1}.$ Now, we have
		\begin{align*}
			&d_H(\mathcal{M}(\mu),\mathcal{M}(\nu) )\\&=\sup \Big\{\Big|\int f d\mathcal{M}(\mu)(t,\boldsymbol{z}) - \int f d\mathcal{M}(\nu)(t,\boldsymbol{z})\Big|: \mbox{Lip}(f) \le 1 \Big\}\\&=\sup \Big\{\Big|\sum_{k \in T} p_k\int f d \mu \circ W_k^{-1}(t,\boldsymbol{z}) - \sum_{k \in T} p_k \int f d\nu \circ W_k^{-1}(t,\boldsymbol{z})\Big|: \mbox{Lip}(f) \le 1 \Big\}\\&=\sup \Big\{\Big|\sum_{k \in T} p_k\int f \circ W_k(t,\boldsymbol{z}) d \mu(t,\boldsymbol{z})  - \sum_{k \in T} p_k \int f\circ W_k(t,\boldsymbol{z}) d\nu(t,\boldsymbol{z}) \Big|: \mbox{Lip}(f) \le 1 \Big\}\\& =  \sum_{k \in T} p_k \sup \Big\{ C_k \Big|\int  \frac{1}{C_k} f \circ W_k(t,\boldsymbol{z}) d \mu(t,\boldsymbol{z})  -  \int \frac{1}{C_k} f\circ W_k(t,\boldsymbol{z}) d\nu(t,\boldsymbol{z}) \Big|: \mbox{Lip}(f) \le 1 \Big\}\\& \le   \sum_{k \in T} p_k C_{\max} \sup \Big\{  \Big|\int  \frac{1}{C_k} f \circ W_k(t,\boldsymbol{z}) d \mu(t,\boldsymbol{z})  -  \int \frac{1}{C_k} f\circ W_k(t,\boldsymbol{z}) d\nu(t,\boldsymbol{z}) \Big|: \mbox{Lip}(f) \le 1 \Big\}\\&\leq  \sum_{k \in T} p_k C_{\max} d_H(\mu,\nu)\\& = C_{\max} d_H(\mu,\nu).
		\end{align*}
		Since $C_{\max}=\max\{C_k: k \in T\}< 1$, the mapping $\mathcal{M}$ is a contraction. Now, the Banach fixed point theorem gives a unique probability measure $\mu_*$ such that $\mu_*=\sum_{k \in T}p_k\mu_* \circ W_k^{-1}. $  It remains to show that $\text{supp}~\mu_*=G(h)$.  For this, we first prove that $\text{supp}~\mu_* \subseteq \cup_{k\in T} W_k (\text{supp}  ~\mu_*).$ On applying $\mu_*=\sum_{k\in T} p_k\mu_* \circ W_k^{-1}$ to $\cup_{k\in T} W_k (\text{supp}~\mu_*),$ we obtain  
		
		\begin{align*}
			\mu_* (\cup_{k\in T} W_k (\text{supp}~\mu_*))&=\sum_{k\in T} p_k \mu_* \circ W_k^{-1}(\cup_{k\in T} W_k (\text{supp}~\mu_*)))\\&
			=\sum_{k\in T} p_k \mu_* (W_k^{-1}(\cup_{k\in T} W_k (\text{supp}~\mu_*)))\\&\ge \sum_{k\in T} p_k \mu_* (W_k^{-1}( W_k (\text{supp}~ \mu_*)))\\&\ge \sum_{k\in T} p_k \mu_* (\text{supp}~ \mu_*)\\&=1.
		\end{align*}
		The above yields that $\text{supp}~\mu_* \subseteq \cup_{k\in T} W_k (\text{supp}~  \mu_*).$ 
		Now, we show that $\cup_{k\in T} W_k (\text{supp}~\mu_*) \subseteq \text{supp}~\mu_* .$ On applying $\mu_*=\sum_{k\in T} p_k\mu_* \circ W_k^{-1}$ to $\text{supp}~\mu_*$, we get 
		$$1=\mu_*  (\text{supp}~\mu_*)=\sum_{k \in T} p_k \mu_* (W_k^{-1} (\text{supp}~\mu_*)) \le \sum_{k \in T} p_k
		=1.$$
		
		Here, by using  the fact that $ \sum_{k \in T} p_k \mu_* (W_k^{-1} (\text{supp}~ \mu_*))=1$  and $\sum_{k \in T} p_k=1$, we have $ \mu_* (W_k^{-1} (\text{supp}~ \mu_*))=1$ for all $k\in T.$ Hence, $\text{supp}~\mu_* \subseteq W_k^{-1}(\text{supp}~ \mu_*)$, which implies that $W_k (\text{supp}~\mu_*) \subseteq (\text{supp}~\mu_*)$ for all $k\in T$. So, $\cup_{k \in T} W_k (\text{supp}~\mu_*) \subseteq (\text{supp}~\mu_*) $. Hence, $\text{supp}~\mu_* = \cup_{k\in T} W_k (\text{supp}~  \mu_*).$ This completes the proof.
	\end{proof}
	

	In the next theorem, we obtain bounds for the Hausdorff dimension of associated invariant measures supported on the graph of vector-valued FIF.
	\begin{theorem}
		Let $\mathcal{I}:=\{J\times \mathbb{R}^M;~~W_k :k\in T\}$ be the IFS as defined earlier such that $$  c_k d((t_1,\boldsymbol{z_1}),(t_2,\boldsymbol{z_2}) ) \le d(W_k(t_1,\boldsymbol{z_1}) , W_k(t_2,\boldsymbol{z_2})) \le C_k d((t_1,\boldsymbol{z_1}),(t_2,\boldsymbol{z_2})) ,$$ where  $(t_1,\boldsymbol{z_1}),(t_2,\boldsymbol{z_2}) \in J\times \mathbb{R}^M$ and $0 < c_k \le C_k < 1 ~~ \forall~~ k \in T .$ Also let $(p_1,p_2,\cdots,p_{N-1})$ be a probability vector corresponding to the IFS $\mathcal{I}$. Then $r \le  \dim_H(\mu_*) \le R  ,$ where $\mu_*$ is an invariant measure corresponding to the IFS $\mathcal{I}$ and $r$, $R$ are given by $ \sum\limits_{k\in T} c_k^{r} =1$ and $ \sum\limits_{k\in T} C_k^{R} =1$, respectively.
	\end{theorem}
	\begin{proof}
		First, we recall the definition of (lower) Hausdorff dimension of measure $\mu_*$
		$$\dim_H(\mu_*)=\inf\{\dim_H(A): A\in \mathcal{B}(J\times \mathbb{R}^M)~~\text{and}~~~\mu_*(A)>0\},$$
		where $\mathcal{B}(J\times \mathbb{R}^M)$ denotes the set of all Borel-subsets of $J\times \mathbb{R}^M.$ Theorem \ref{th3.9} yields that $\mu_*(G(h))=1.$ Therefore, by the definition of Hausdorff dimension of measure $\mu_*$ and Theorem \ref{th3.6}, we get 
		$$\dim_H(\mu_*)\leq R.$$ For determining a lower bound of $\dim_H(\mu_*)$, let $B$ be any Borel-subset of $J\times \mathbb{R}^M$ such that $\dim_H(B)<r\leq \dim_H(G(h)).$ Our aim is to show that $\mu_*(B)=0$. We prove it by contradiction. Assume that $\mu_*(B)>0$. We define a set $U=\frac{B}{\gamma},$ where $\gamma=\mu_*(B).$ So, it is clear that $\mu_*(U)=1$. Since $G(h)$ is the support of $\mu_*$, therefore,  $G(h)\subset U.$ Thus, by the  monotonic property of Hausdorff dimension, $\dim_H(G(h))\leq \dim_H(U).$ Theorem \ref{th3.6} gives that $r\leq \dim_H(U) .$ Since the Hausdorff dimension of a set does not change by scaling by a non-zero constant, therefore, we get $r\leq \dim_H(B),$ which is a contradiction. Thus, we have  $\mu_*(B)=0.$ Therefore, by the definition of $\dim_H(\mu_*)$, we deduce that $r\leq\dim_H(\mu_*)$. Thus, the proof is completed.
	\end{proof}
	
	In the upcoming result, we establish a more general upper bound for the Hausdorff dimension of invariant measure in terms of the probability vector and the contraction ratios. 
	\begin{theorem}\label{dimM}
		Let $\mathcal{I}:=\{J\times \mathbb{R}^M;~~W_k :k\in T\}$ be the IFS as defined earlier such that $$ d(W_k(t_1,\boldsymbol{z_1}) , W_k(t_2,\boldsymbol{z_2})) \le C_k d((t_1,\boldsymbol{z_1}),(t_2,\boldsymbol{z_2})) ,$$ where  $(t_1,\boldsymbol{z_1}),(t_2,\boldsymbol{z_2}) \in J\times \mathbb{R}^M$ and $0 < C_k < 1 ~ \forall~ k \in T .$ Also, let $(p_1,p_2,\cdots,p_{N-1})$ be a probability vector corresponding to the IFS $\mathcal{I}$. Then $$\dim_H(\mu_*)\leq \frac{\sum\limits_{k\in T}p_k\log{p_k}}{\sum\limits_{k\in T}p_k\log{C_k}},$$ where $\mu_*$ is the invariant measure corresponding to the IFS $\mathcal{I}.$
	\end{theorem}
	
	\begin{proof}
		Let $(t,h(t))\in G(h)$ and let $B((t,h(t)),r)$  be a ball of radius $r$ with center $(t,h(t))$. There is $\sigma=(\sigma_1,\sigma_2,\cdots)\in \Omega$  such that $\sigma$ is a code of $(t,h(t)).$  Let $q$ be the smallest integer such that 
		$$C_{\sigma_1}\cdot C_{\sigma_2}\cdots C_{\sigma_q} ~~\text{diam}(G(h))< r\leq C_{\sigma_1}\cdot C_{\sigma_2}\cdots C_{\sigma_{q-1}}~~\text{diam}(G(h))$$
		From the definition of $q$, it is clear that  
		$$ W_{\sigma_1}\circ W_{\sigma_2}\circ \cdots W_{\sigma_q}(G(h))\subset B((t,h(t)),r). $$
		Also, we have   
		\begin{align}
			r C_{\min}{\beta}^{-1}&\leq \prod_{i=1}^{q}C_{\sigma_i}<  r{\beta}^{-1}
		\end{align}
		where $C_{\min}=\min\{C_k: k\in T\}$ and $\beta = \text{diam}{(G(h))}.$ Also
		
		$$\mu_*( B((t,h(t)),r))\geq p_{\sigma_1}\cdot p_{\sigma_2}\cdots p_{\sigma_q}=\prod_{i=1}^{q}p_{\sigma_i}.$$
		Then by the previous inequalities, we get
		\begin{equation}\label{e3.2}
			\mu_*( B((t,h(t)),r))\geq \frac{\prod\limits_{i=1}^{q}p_{\sigma_i}}{\prod\limits_{i=1}^{q}C_{\sigma_i}} r C_{\min}{\beta}^{-1}= C_* \cdot r \frac{\prod\limits_{i=1}^{q}p_{\sigma_i}}{\prod\limits_{i=1}^{q}C_{\sigma_i}},
		\end{equation}
		where $C_*= C_{\min}{\beta}^{-1}.$
		
		Now, we define two functions $g_1,g_2: \Omega\to \mathbb{R}$ such that
		$$ g_1(w_1,w_2,\cdots)=\log p_{w_1}~~~\text{and}~~~ g_2(w_1,w_2,\cdots)= \log C_{w_1},$$
		for $w=(w_1,w_2,\cdots)\in \Omega.$ Then, by the application of Birkhoff's Ergodic Theorem, for $\mu$-a.e. $w\in \Omega,$ 
		\begin{align}\label{eq3.3}
			\lim\limits_{n\to \infty}\frac{1}{n}\sum\limits_{i=0}^{n-1}g_1(S_i(w))&=\int_{\Omega}g_1(w)d\mu,\\
			\lim\limits_{n\to \infty}\frac{1}{n}\sum\limits_{i=0}^{n-1}g_2(S_i(w))&=\int_{\Omega}g_2(w)d\mu,
		\end{align}
		where $S_i: \Omega\to \Omega$ is a map defined by $S_i(w)=(w_{i+1},w_{i+2},\cdots)$. First, we simplify terms in above equalities
		$$ \frac{1}{n}\sum\limits_{i=0}^{n-1}g_1(S_i(w))=\frac{1}{n}\sum\limits_{i=0}^{n-1}g_1(w_{i+1},w_{i+2},\cdots) 
		= \frac{1}{n}\sum\limits_{i=0}^{n-1}\log p_{w_{i+1}}
		= \frac{1}{n}\log \prod\limits_{i=1}^{n}p_{w_i}.$$
		
		Also, 
		\begin{align*}
			\int_{\Omega}g_1(w)d\mu=\sum\limits_{k\in T} \int_{[k]}g_1(w)d\mu =  \sum\limits_{k\in T}\mu([k])\log p_{k}= \sum\limits_{k\in T} p_k \log p_k.
		\end{align*}
		Thus, by the above estimation and by using \eqref{eq3.3}, for $\mu$-a.e. $w\in \Omega$, we have 
		\begin{align}\label{eq3.5}
			\lim\limits_{n\to \infty}\frac{1}{n}\log \prod\limits_{i=1}^{n}p_{w_i}=\sum\limits_{k\in T} p_k \log p_k.
		\end{align}
		Similarly for $g_2$, we have
		\begin{align}\label{eq3.6}
			\lim\limits_{n\to \infty}\frac{1}{n}\log \prod\limits_{i=1}^{n}C_{w_i}=\sum\limits_{k\in T} p_k \log C_k.
		\end{align}
		By the definition of $q$, it is clear that $q\to \infty$ if $r\to 0$. Hence, for $\mu$-a.e. $w\in \Omega$, we have
		\begin{align}\label{eq3.7}
			\lim\limits_{r\to 0}\frac{\log r}{q}= \lim\limits_{q\to \infty} \frac{1}{q}\log \prod\limits_{i=1}^{q}C_{w_i}=\sum\limits_{k\in T} p_k \log C_k.
		\end{align}
		Using equations \eqref{eq3.5}, \eqref{eq3.6} and \eqref{eq3.7}, we conclude that for $\mu$-a.e. $w\in \Omega$
		\begin{align}\label{e3.8}
			\lim\limits_{r\to 0}\frac{\log \prod\limits_{i=1}^{q}\frac{p_{w_i}}{C_{w_i}}}{\log r} =\frac{\sum\limits_{k\in T} p_k \log p_k}{\sum\limits_{k\in T} p_k \log C_k}-1 .
		\end{align}
		Combining inequality \eqref{e3.2} and equation \eqref{e3.8}, we have, for $\mu_*$-almost all $(t,\boldsymbol{z})\in J\times \mathbb{R}^M$
		\begin{align*}
			\limsup_{r\to 0}\frac{\log\mu_*(B((t,\boldsymbol{z}),r))}{\log r}\leq \frac{\sum\limits_{k\in T} p_k \log p_k}{\sum\limits_{k\in T} p_k \log C_k}. 
		\end{align*}
		Thus, by Proposition \ref{loc}, we get 
		$$\dim_H(\mu_*)\leq \frac{\sum\limits_{k\in T}p_k\log{p_k}}{\sum\limits_{k\in T}p_k\log{C_k}}.$$
		Thus, the proof is done.
	\end{proof}

	The H\"{o}lder space is defined as follows:$$ \mathcal{HC}^{\sigma}(J ) := \{g:J \rightarrow \mathbb{R}^M: ~\text{g is H\"{o}lder continuous with exponent}~ \sigma \} .$$
	Note that $(\mathcal{HC}^{\sigma}(J),\|.\|_\mathcal{HC})$ is a Banach space, where $ \|g\|_{\mathcal{HC}}:= \|g\|_{\infty} +[g]_{\sigma}$ and $$[g]_{\sigma} = \sup_{t_1\ne t_2} \frac{\|g(t_1)-g(t_2)\|}{|t_1-t_2|^{\sigma}}$$
	
	
	In the next result, we describe some conditions under which vector-valued FIF  belongs to the H\"older space.
	\begin{theorem}\label{BBVL3}
		Let $ q_k \in  \mathcal{HC}^{\sigma}(J )$ for each $k \in T $. Set $a_{\min}:= \min\{|a_k|: k \in T \}$ and $\alpha_{\max}:= \max\{|\alpha_k|: k \in T \}$. If $  \frac{\alpha_{\max}}{a_{\min}^\sigma}< 1 $, then the fractal function
		$h$ is H\"{o}lder continuous with exponent $\sigma$.
	\end{theorem}
	\begin{proof}
		Let us define $ \mathcal{HC}^{\sigma}_0(J ):= \{ f \in \mathcal{HC}^{\sigma}(J ): f(x_1)= \boldsymbol{y}_1, ~f(x_N)= \boldsymbol{y}_N \}.$ 
		By basic real analysis technique, we may see that $\mathcal{HC}^{\sigma}_0(J )$ is a closed subset of $\mathcal{HC}^{\sigma}(J).$ Since $(\mathcal{HC}^{\sigma}(J),\|.\|_\mathcal{HC})$ is a Banach space, it implies that $\mathcal{HC}^{\sigma}_0(J )$  will be a complete metric space with respect to metric induced by $\|.\|_\mathcal{HC}.$ We define RB operator $S: \mathcal{HC}^{\sigma}_0(J) \rightarrow \mathcal{HC}^{\sigma}_0(J )$ by $$ (Sf)(t)=\alpha_k f(P_k^{-1}(t))+ ~q_k(P_k^{-1}(t))  $$
		$\forall~~ t \in J_k ,$ where $k \in T.$ We shall show that $S$ is well-defined and is a contraction map on $\mathcal{HC}^{\sigma}_0(J)$.
		\begin{equation*}
			\begin{split}
				[Sf]_\sigma   = &\max_{k  \in T} \sup_{t_1 \ne t_2, t_1,t_2  \in J_k} \frac{\|Sf(t_1)-Sf(t_2)\|}{|t_1-t_2|^{\sigma}}\\
				\le&  \max_{k  \in T} \Bigg[ \sup_{t_1 \ne t_2, t_1,t_2 \in J_k} \frac{\|\alpha_k f(P_k^{-1}(t_1))-\alpha_k f(P_k^{-1}(t_2))\|}{|t_1-t_2|^{\sigma}}\\
				& +  \sup_{t_1\ne t_2, t_1,t_2 \in J_k} \frac{ \Big\|q_k(P_k^{-1}(t_1))-q_k(P_k^{-1}(t_2))\Big\|}{|t_1-t_2|^{\sigma}}\Bigg]\\&\leq\frac{\alpha_{\max} [f]_\sigma}{a_{
						\min}^\sigma} +\frac{[q]_\sigma}{a_{\min}^\sigma},
			\end{split}
		\end{equation*}
		where $[q]_{\sigma}= \max\limits_{k \in T} ~~[q_k]_\sigma$. Let $f, g \in \mathcal{HC}^{\sigma}_0(J )$. Then we have
		\begin{equation*}
			\begin{aligned}
				\|Sf -Sg\|_{\mathcal{HC}} &= \|Sf -Sg\|_{\infty} + [Sf-Sg]_{\sigma}\\
				&\le \alpha_{\max} \|f -g\|_{\infty} + \frac{\alpha_{\max} [f-g]_{\sigma}}{a_{\min}^{\sigma}} \\
				&\le \frac{\alpha_{\max}}{a_{\min}^{\sigma}} \| f-g\|_{\mathcal{H}}.
			\end{aligned}
		\end{equation*}
		This implies that $S$ is a well-defined map on $\mathcal{HC}^{\sigma}_0(J )$.    
		Since $\frac{\alpha_{\max}}{a_{\min}^{\sigma}} < 1$, it follows that $S$ is a contraction map on $ \mathcal{HC}^{\sigma}_0(J ).$ Using Banach contraction mapping theorem, $S$ has a unique fixed point $ h \in \mathcal{HC}^{\sigma}_0(J )$. This completes the proof.
	\end{proof}

	In the next theorem, we calculate fractal dimensions of the graph of the vector-valued FIF.
	\begin{theorem}
		Let $\sigma\in (0,1]$ and   $q_{k}\in \mathcal{HC}^\sigma(J)$ such that
		$$ \|q_k(t_1) -q_k(t_2)\| \le C_q |t_1-t_2 |^{\sigma}$$
		for all $t_1,t_2\in J, k\in T$ and $C_q>0$. Let for each $i\in \{1,2,\cdots,M\} $, $q_{k,i}$ and $h_i$ be the $i$th component of $q_k$ and $h$, respectively. Moreover, if we assume that for each $k \in T$, there are constants $C_{k,i}>0$ and $\delta_0>0$ such that for each $t_1\in J$ and $0<\delta \leq \delta_0 $ there exists $t_2\in J$ such that $|t_1-t_2|\leq \delta$  and
		$$|q_{k,i}(t_1)-q_{k,i}(t_2)|\geq C_{k,i}|t_1-t_2|^\sigma~~~\forall~~i\in \{1,2,\cdots,M\}. $$
		Then we have 
		$$1 \le \dim_H\big(G{(h_i)}\big) \le \dim_B\big(G{(h_i)}\big) = 2 - \sigma~~~\forall~~i=\{1,2,\cdots,M\},$$
		provided  $\frac{\alpha_{\max}}{a_{\min}^\sigma}<1$ and $C_{0}a_{\min}^\sigma>\alpha_{\max}a_{\max}^\sigma h_q$ holds, where $C_0=\min\limits_{k\in T}\{\min\limits_{1\leq i\leq M} C_{k,i}\},$ and $a_{\max}=\max\{|a_k| : k\in T\}$. Moreover, 
		$\dim_B(G(h))\geq 2-\sigma.$
	\end{theorem}
	\begin{proof}
		Since $\frac{\alpha_{\max}}{a_{\min}^\sigma}<1$, by Theorem \ref{BBVL3}, we deduce that the fractal function $h\in \mathcal{HC}^{\sigma}(J).$ For $t_1,t_2\in J$, we have
		$$ |h_i(t_1)-h_i(t_2)|\leq \|h(t_1)-h(t_2)\|\leq h_{q} |t_1-t_2|^\sigma, $$  holds for each $i\in \{1,2,\cdots,M\}$ and for some $h_q>0.$ First, we will prove that 
		$$\overline{\dim}_{B}G(h_i)\leq 2-\sigma~~\text{for all}~~i\in \{1,2,\cdots,M\}.$$
		Let $m$ be the smallest natural number greater than or equal to $\frac{1}{\delta}$ and $N_{\delta}(G(h_i))$ be the smallest number of $\delta$-mess that can intersect with $G(h_i)$. Then we have
		\begin{align*}
			N_{\delta}(G(h_i))&\leq 2m+\delta^{-1}\sum\limits_{j=0}^{m-1}R_{h_i}[j\delta,(j+1)\delta]\\
			&\leq 2(\delta^{-1}+1)+ \sum\limits_{j=0}^{m-1}h_{q}\delta^{\sigma-1}\\
			&\leq \delta^{\sigma-2} (4+2mh_{q}).
		\end{align*}
		From the above, we conclude that 
		\begin{equation}\label{eq3.2}
			\overline{\dim}_B\big(G(h_i)\big) =\varlimsup_{\delta \rightarrow 0} \frac{\log N_{\delta}(G(h_i))}{- \log \delta}\le 2- \sigma~~~~~\forall~~~ i=1,2,\cdots,M.
		\end{equation}
		Next, we will give a lower bound of lower box dimension of the graph of $h_i$ for each $i\in \{1,2,\cdots,M\}.$ Self-referential equation yields that 
		$$h_i(t)=\alpha_kh_i(P_k^{-1}(t))+q_{k,i}(P_k^{-1}(t))~~~\forall~~ i\in\{1,2,\cdots,M\}$$ for each $t\in J_k$ and $k\in T$ .
		Let $\delta_1=\delta_{0}a_{\min}$ . Let $\delta\leq \delta_1$ and $t_1,t_2\in J_k$ such that $|t_1-t_2|\leq \delta$. Then we have
		\begin{align*}
			|h_i(t_1)-h_i(t_2)|&=|\alpha_kh_i(P_K^{-1}(t_1))+q_{k,i}(P_k^{-1}(t_1))-\alpha_kh_i(P_K^{-1}(t_2))-q_{k,i}(P_k^{-1}(t_2))| \\
			&\geq |q_{k,i}(P_k^{-1}(t_1))-q_{k,i}(P_k^{-1}(t_2))|-|\alpha_kh_i(P_K^{-1}(t_1))-\alpha_kh_i(P_K^{-1}(t_2))|\\
			&\geq C_{0} a_{\max}^{-\sigma}|t_1-t_2|^\sigma -\alpha_{\max} a_{\min}^{-\sigma} h_{q} |t_1-t_2|^\sigma \\
			&=(C_{0} a_{\max}^{-\sigma}-\alpha_{\max} a_{\min}^{-\sigma} h_{q}) |t_1-t_2|^\sigma\\
			&=R_0 |t_1-t_2|^\sigma,
		\end{align*}
		where $R_0=C_{0} a_{\max}^{-\sigma}-\alpha_{\max} a_{\min}^{-\sigma}$. By our assumption, it is clear that $R_0>0.$ Let $\delta \leq \delta_1$. Then we have 
		$$  N_\delta(G(h_i))\geq
		\sum_{j=0}^{m-1}  { \delta^{-1}} R_{h_i}[j\delta,(j+1)\delta]  \ge  \sum_{j=0}^{m-1 }  {R_0  \delta^{ \sigma -1}}
		\geq  R_0 \delta^{ \sigma-2}.$$
		
		By the above, we conclude that
		\begin{equation}\label{e3.3}
			\underline{\dim}_B\big(G(h_i))\big) =\varliminf_{\delta \rightarrow 0}\frac{ \log\Big( N_{\delta}(G(h_i))\Big)}{- \log (\delta)}
			\geq
			2- \sigma
			~~~~~~\forall~~i\in \{1,2,\cdots,M\}.
		\end{equation}
		Using inequalities \eqref{eq3.2}, \eqref{e3.3} and Lemma \ref{new222}, we get our assertion.
	\end{proof}

	\begin{definition}
		A vector-valued function $ g:J \rightarrow \mathbb{R}^M$ is said to be of bounded variation if the total variation $V(g,J) $ of $g$, defined by $$V(g,J)= \sup_{\Delta=(t_0,t_1, \dots,t_l) ~~\text{partition of } J}~ \sum_{j=0}^{l-1} \|g(t_j)-g(t_{j+1})\|,$$ is finite.
		The space of all vector-valued bounded variation functions on $J,$ denoted by $\mathcal{BV}(J,\mathbb{R}^M),$ forms a Banach space when endowed with the norm
		$\|g\|_{\mathcal{BV}}:= \|g(x_1)\|+ V(g,J).$
		
	\end{definition}
	In the following theorem, we find out some conditions under which vector-valued FIF becomes a bounded variation function and the Hausdorff dimension and the box dimension of vector-valued FIF is $1.$
	
	\begin{theorem}\label{th3.14}
		Let $q_k \in \mathcal{BV}(J,\mathbb{R}^M)~~\text{for all}~~k\in T$. Let $\alpha_{\max}=\max\{|\alpha_k| : k\in T\}$. If $\alpha_{\max}< \frac{1}{(N-1)},$ then $h \in \mathcal{BV}(J,\mathbb{R}^M)$. Moreover, $\dim_H(G(h))=\dim_B(G(h))=1 ,$ and either $\dim_H(\mu_*)=1$ or the $1$-dimensional Hausdorff measure $(\mathcal{H}^1)$ is absolutely continuous with respect to $\mu_*.$
	\end{theorem}
	\begin{proof}
		First, we define a set $\mathcal{BV}_0(J,\mathbb{R}^M)=\{f\in \mathcal{BV}(J,\mathbb{R}^M): f(x_1)=\boldsymbol{y}_1, f(x_N) =\boldsymbol{y}_N\}.$ It is easily seen that  $\mathcal{BV}_0(J,\mathbb{R}^M)$ is a closed subset of $\mathcal{BV}(J,\mathbb{R}^M)$. Since $\mathcal{BV}(J,\mathbb{R}^M)$ is a Banach space, $\mathcal{BV}_0(J,\mathbb{R}^M)$ is a complete metric space with respect to the metric induced by norm $\|f\|_{\mathcal{BV}}:= \|f(x_1)\|+ V(f,J).$ We define an operator $S: \mathcal{BV}_0(J,\mathbb{R}^M)\to \mathcal{BV}_0(J,\mathbb{R}^M) $ by
		$$ (Sf)(t)=\alpha_k f(P_k^{-1}(t))+ ~q_k(P_k^{-1}(t))  $$
		$\forall~~ t \in J_k $ and $k \in T.$ Consider a partition $\Delta=\{t_0,t_1,\cdots,t_l\}$ of $J_k$, where $k\in T$ and $l\in \mathbb{N}.$ For $f,g\in \mathcal{BV}_0(J,\mathbb{R}^M)$, we have
		\begin{align*}
			\|(Sf-Sg)(t_j)-(Sf-Sg)(t_{j+1})\|&=\|\alpha_k (f-g)(P_k^{-1}(t_j))-\alpha_k (f-g)(P_k^{-1}(t_{j+1}))\|\\
			&=|\alpha_k|\|(f-g)(P_k^{-1}(t_j))- (f-g)(P_k^{-1}(t_{j+1}))\|\\
			&\leq \alpha_{\max} \|(f-g)(P_k^{-1}(t_j))- (f-g)(P_k^{-1}(t_{j+1}))\|.
		\end{align*}   
		In the above, by taking sum from $j=0$ to $j=l-1$, we get
		\begin{align*}
			\sum\limits_{j=0}^{l-1}\|(Sf-Sg)&(t_j)-(Sf-Sg)(t_{j+1})\|\\&\leq \alpha_{\max} \sum\limits_{j=0}^{l-1}\|(f-g)(P_k^{-1}(t_j))- (f-g)(P_k^{-1}(t_{j+1}))\|\\
			& \leq \alpha_{\max} \|f-g\|_{\mathcal{BV}}.
		\end{align*}  
		The above inequality is true for any partition $\Delta$ of $J_k$  and $k\in T$. Thus, we deduce that
		$$\|Sf-Sg\|_{\mathcal{BV}}\leq \alpha_{\max}(N-1)\|f-g\|_{\mathcal{BV}}.$$
		So, $S$ is a well defined operator on $\mathcal{BV}_0(J,\mathbb{R}^M)$ and by our hypothesis, it is clear that $\alpha_{\max}(N-1)<1.$ Therefore, $S$ is a contraction map on the complete metric space $\mathcal{BV}_0(J,\mathbb{R}^M)$. By virtue of the Banach contraction mapping theorem, $S$ has a unique fixed point $h\in \mathcal{BV}_0(J,\mathbb{R}^M).$
		\par
		Suppose $\mathcal{H}^1$ is not absolutely continuous with respect  to $\mu_*$. Then there exists a set $A$ such that $ \mu_*(A)=0$ but $\mathcal{H}^1(A)>0.$ This yields that $\dim_H(\mu_*) \ge 1.$ Since $\dim_H(\mu_*) \le \dim_H(G(h))=1,$ we have $\dim_H(\mu_*) =1.$
		This completes the proof.   
	\end{proof}
	
	Next,  we define  the  absolutely continuous function space for vector-valued functions.
	The space of all absolutely continuous vector-valued functions on $J$ is denoted by $\mathcal{AC}(J,\mathbb{R}^M)$.
	We define a norm on $\mathcal{AC}(J,\mathbb{R}^M)$ as follows:
	$$\|f\|_{\mathcal{AC}}=\|f(x_1)\| +\int_{x_1}^{x_N}\|f'(x)\| dx,$$
	where $f\in \mathcal{AC}(J,\mathbb{R}^M)$. Then $\mathcal{AC}(J,\mathbb{R}^M)$ forms a Banach space under the norm $\|\cdot\|_{\mathcal{AC}}$.
	
	In the upcoming result, we obtain some conditions under which vector-valued FIF  belongs to the absolutely continuous function space, and also compute fractal dimension of the graph of vector-valued FIF.
	\begin{theorem}
		Let $ q_k \in \mathcal{AC}(J,\mathbb{R}^M)$ for each $k \in T $. Set $a_{\min}:= \min\{|a_k|: k \in T \}$ and $\alpha_{\max}:= \max\{|\alpha_k|: k \in T \}$. If $  {\alpha_{\max}}< \frac{a_{\min}}{(N-1)} $, then the fractal function
		$h$ is also in $\mathcal{AC}(J,\mathbb{R}^M)$. Moreover, $\dim_H(G(h))=\dim_B(G(h))=1 .$
	\end{theorem}
	\begin{proof}
		One can easily prove this theorem by using the similar arguments as in the proof of Theorem \ref{th3.14}.
	\end{proof}
	
	Next, for $\alpha\geq1$, we define a function space as follows:
	$$V_\alpha(J,\mathbb{R}^M)= \{f\in C(J,\mathbb{R}^M): \|f\|_{\alpha}<\infty\}.$$ Then, by Lemma 3.1  in \cite{Fal&Far}, $(V_\alpha(J,\mathbb{R}^M),\|.\|_\alpha)$ is a complete metric space, where
	$$\|f\|_{\alpha}=\|f\|_{\infty}+\sup\limits_{n\in \mathbb{N}}\frac{\sum_{|O|=2^{-n}}R_f{(O)}}{2^{n{(\alpha-1)}}},$$ 
	where $R_f(O)=\sup\limits_{t_1,t_2\in O}\|f(t_1)-f(t_2)\|$ for $O\subseteq J.$
	\begin{lemma} \label{RFD}
		If $|P_k(J)|=\frac{1}{2^{r_k}}$ for some $r_k\in \mathbb{N}$ with $\sum\limits_{k\in T}\frac{1}{2^{r_k}}=1,$ then for $n\geq \max\limits_{k\in T}\{r_k\}, $ we have 
		$$\sum_{|O|=2^{-n}}R_f{(O)}=\sum\limits_{k\in T}\sum_{|O|=2^{-n}, O\subset P_k(J)}R_f{(O)}$$
	\end{lemma}
	\begin{proof} One can easily prove this lemma.
		
	\end{proof}
	\begin{theorem}\label{hinV}
		Let $q_k\in V_\alpha(J,\mathbb{R}^M) ~~\forall~~k\in T$ and let $|P_k(J)|=\frac{1}{2^{r_k}}$ for some $r_k\in \mathbb{N}$ with $\sum\limits_{k\in T}\frac{1}{2^{r_k}}=1.$ Then 
		$h\in V_\alpha(J,\mathbb{R}^M),$ provided $\sum\limits_{k\in T}|\alpha_k|<1.$
	\end{theorem}
	\begin{proof}
		Let $V_\alpha^0(J)=\{f\in V_\alpha(J,\mathbb{R}^M): f(x_1)=\boldsymbol{y}_1 ~~\text{and}~~f(x_N)=\boldsymbol{y}_N\}.$ From the definition of $V_\alpha^0(J),$ it is clear that $V_\alpha^0(J)$ is a closed subset of $V_\alpha(J,\mathbb{R}^M).$ Therefore,  $V_\alpha^0(J)$ is a complete metric space with respect to metric induced by norm $\|.\|_\alpha.$ We define RB operator $S: V_\alpha^0(J)\to V_\alpha^0(J)$ by 
		$$(Sf)(t)=\alpha_k f(P_k^{-1}(t))+ ~q_k(P_k^{-1}(t))  $$
		$\forall~~ t \in J_k $ and $k \in T.$ In the light of the assumptions taken, it is clear that $S$ is well-defined. Let $f,g\in V_\alpha^0(J).$ By applying lemma \ref{RFD}, we have 
		\begin{align*}
			\|Sf-Sg\|_\alpha&= \|Sf-Sg\|_{\infty}+\sup\limits_{n\in \mathbb{N}}\frac{\sum_{|O|=2^{-n}}R_{(Sf-Sg)}{(O)}}{2^{n{(\alpha-1)}}}\\
			&\leq \alpha_{\max} \|f-g\|_{\infty}+ \sum\limits_{k\in T}|\alpha_k| \sup\limits_{n\in \mathbb{N}}\frac{\sum_{|O|=2^{-n}}R_{(f-g)}{(O)}}{2^{n{(\alpha-1)}}}\\
			&\leq \sum\limits_{k\in T}|\alpha_k|\bigg(\|f-g\|_{\infty}+\sup\limits_{n\in \mathbb{N}}\frac{\sum_{|O|=2^{-n}}R_{(f-g)}{(O)}}{2^{n{(\alpha-1)}}}\bigg)\\
			&= \sum\limits_{k\in T}|\alpha_k| \|f-g\|_{\alpha}.
		\end{align*}
		Since $\sum\limits_{k\in T}|\alpha_k|<1$, $S$ is a contraction map on  $V_\alpha^0(J).$  Applying the Banach contraction mapping theorem, operator $S$ has a unique fixed point $h$ in $V_\alpha^0(J).$ Furthermore, the function  $h$ satisfies self-referential equation, that is, 
		$$h(t)=\alpha_k h(P_k^{-1}(t))+ ~q_k(P_k^{-1}(t)),  $$
		$\forall~~ t \in J_k $ and $k \in T.$ This completes the proof.
		
	\end{proof}
	
	\begin{remark}
		Let us consider functions $\textbf{f},\textbf{g}:[0,1] \to \mathbb{R}^2$ such that $\textbf{f}=(h_1,0),\textbf{g}=(0,h_2)$, where $h_1$ and $h_2$ are coordinate functions of the Peano space filling curve. It is simple to deduce that ( see \cite{Kono}) $$\overline{\dim}_B\big(G({\textbf{f}})\big)=\overline{\dim}_B\big(G({h_1})\big) = 1.5~ \text{and} ~\overline{\dim}_B\big(G({\textbf{g}})\big)=\overline{\dim}_B\big(G({h_2})\big) = 1.5,$$ and 
		$$\overline{\dim}_B\bigg(G(\textbf{f}+\textbf{g})\bigg)=\overline{\dim}_B\bigg(G(h_1,h_2)\bigg)=2.$$
		So, we conclude that the inequality 
		\[
		\overline{\dim}_B\big(G({\textbf{f}+\textbf{g}})\big) \le \max\{\overline{\dim}_B\big(G({\textbf{f}})\big),\overline{\dim}_B\big(G({\textbf{g}})\big)\}=1.5,
		\]
		does not hold for vector-valued functions in contrast to its real-valued setting \cite[Lemma 2.1]{Fal&Far}. It hints that the space $\{f \in \mathcal{C}([a,b], \mathbb{R}^M) :\overline{\dim}_B\big(G({\textbf{f}})\big)\le \beta\}$ will not be a vector space for $M \ge 2.$ Thus, the result of Falconer and Fraser \cite{Fal&Far} cannot be generalized to vector-valued continuous functions.
	\end{remark}
	Thus, due to the above remark, the next result is valid only for real-valued setting. In the next result, we determine an upper bound for the upper box dimension of the graph of real-valued fractal interpolation functions.  
	\begin{theorem}
		Let $q_k\in V_{\alpha+\frac{1}{n}}(J,\mathbb{R}) ~~\forall~~k\in T, n\in \mathbb{N}$ and let $|P_k(J)|=\frac{1}{2^{r_k}}$ for some $r_k\in \mathbb{N}$ with $\sum\limits_{k\in T}\frac{1}{2^{r_k}}=1.$ If $\sum\limits_{k\in T}|\alpha_k|<1,$ then $\overline{\dim}_B(G(h))\leq \alpha.$
	\end{theorem}
	\begin{proof}
		In Theorem \ref{hinV}, if we take $N=1$ and replace $\alpha$ by $\alpha+\frac{1}{n}$, then we get $h\in V_{\alpha+\frac{1}{n}}(J,\mathbb{R})$ for all $n\in \mathbb{N}.$ By \cite[Proposition 3.4]{Fal&Far}, we have
		$$\{f\in \mathcal{C}(J,\mathbb{R}):\overline{\dim}_B(G(f))\leq \alpha\}=\bigcap\limits_{n\in \mathbb{N}}V_{\alpha+\frac{1}{n}}(J,\mathbb{R}).$$ 
		Therefore, from the above equality, we can deduce that $\overline{\dim}_B(G(h))\leq \alpha.$ This completes the proof.
	\end{proof}
	\section{Riemann-Liouville fractional integral}\label{sc-4}
	In this Section, first we define  the Riemann-Liouville fractional integral of a vector-valued function.
	\begin{definition}
		Let $\textbf{f}$ be a vector-valued integrable function on a closed interval $[a,b].$ The Riemann-Liouville fractional integral of $\textbf{f}$ is defined as $$ _a \mathfrak{I}^{\beta}\textbf{f}(t)=\Big(~~_a \mathfrak{I}^{\beta}f_1(t),~_a \mathfrak{I}^{\beta}f_2(t),\dots,~~ _a \mathfrak{I}^{\beta}f_M(t)\Big),$$ 
		where
		$$ _a \mathfrak{I}^{\beta}f_i(t)=\frac{1}{\Gamma (\beta)} \int_a ^t (t-\eta)^{\beta-1} f_i(\eta)~\mathrm{d}\eta,~~\text{for}~~ i=1,2, \dots,M,$$ and $\beta>0$.
	\end{definition}
	The following theorem can be deduced from Liang \cite{Liang1}.
	\begin{theorem} \label{MT1}
		For $0 < a <b < \infty$ and $0<\beta<1 .$
		If $\textbf{f} \in \mathcal{C}([a,b],\mathbb{R}^M)\cap \mathcal{BV}([a,b],\mathbb{R}^M)$, then $_a\mathfrak{I}^{\beta}\textbf{f} \in \mathcal{C}([a,b],\mathbb{R}^M)\cap \mathcal{BV}([a,b],\mathbb{R}^M).$
	\end{theorem}
	There are several works on fractal dimension of fractional integral of a real-valued continuous functions, see, for instance, \cite{SS,Liang5,Liang4,Liang EST}. Let us note the following inequality:
	$$ \|\textbf{f}(x) -\textbf{f}(y)\|_2 \le \sqrt{M} \max_{1 \le i \le M}| f_i(x) -f_i(y)| \le \sqrt{M} L_{\textbf{f}} |x-y|^\sigma,$$ 
	where $f_i$ denotes the $i$-th coordinate function of the vector-valued function $\textbf{f}, $  $L_{\textbf{f}}=\max\{L_{f_1},\\L_{f_2}, \dots, L_{f_M}\}$ and $L_{f_i}$ is the H\"older constant of $f_i$. Using \cite{Liang6} and the above inequality, we can immediately obtain the next remark. Hence, we omit the proof.
	\begin{remark}
		Let $\textbf{f}:[a,b]\to \mathbb{R}^M$ be a continuous function.
		\begin{enumerate}
			\item  If $ 0< \beta < 1,$ then $$ 1 \le  \dim_H(\text{G}(_a\mathfrak{I}^{\beta}f_i)) \le \overline{\dim}_B(\text{G}(_a\mathfrak{I}^{\beta}f_i)) \le 2- \beta.$$
			\item If $  \beta \ge 1,$ then $$ \dim_H(\text{G}(_a\mathfrak{I}^{\beta}\textbf{f})) = \dim_B(\text{G}(_a\mathfrak{I}^{\beta}\textbf{f}))=1 .$$
		\end{enumerate}
	\end{remark}
	
	In the following theorem, we prove that the Riemann-Liouville fractional integral of a vector-valued FIF is again a  vector-valued FIF corresponding to different data set and show the existence of a unique Borel probability measure supported on the graph of ${_{x_1} \mathfrak{I}^{\beta}h}.$
	
	\begin{theorem}
		Let $h$ be the FIF determined by the IFS mentioned before in Section \ref{se2}. Then $ _{x_1} \mathfrak{I}^{\beta}h$ is the FIF associated to the the IFS
		$\mathcal{J}_1=\{J\times \mathbb{R}^M: H_1,H_2,\dots,H_{N-1}\}$ with data set $\{(x_i,~~{_{x_1} \mathfrak{I}^{\beta}}h(x_i)): i\in{1,2,\cdots,N} \},$  where $H_k(t,\boldsymbol{z})=(P_k(t),F'_k(t,\boldsymbol{z})) $ for $k\in T$ and $P_k(t)$ and $F'_k(t,\boldsymbol{z})$ are defined as follows
		$$P_k(t)=a_{k}t+d_k,~~~F'_k(t,\boldsymbol{z})= a_k^{\beta} \alpha_k \boldsymbol{z}+ Q_k(t),$$  where $a_k,\alpha_k,d_k$ are defined as above, $Q_k(t):= (Q_{k,1}(t),Q_{k,2}(t), \dots,Q_{k,M}(t))$ is vector-valued continuous function on $J$ and for each $i\in \{1,2,\cdots,M\}$
		$$Q_{k,i}(t)= \frac{1}{\Gamma (\beta)} \int_{x_1} ^{P_k(x_1)} (P_k(t)-\eta)^{\beta-1} h_i(\eta)~\mathrm{d}\eta + \frac{a_k^{\beta} }{\Gamma (\beta)} \int_{x_1} ^{t} (t-\eta)^{\beta-1} q_{k,i}(\eta)~\mathrm{d}\eta.$$ Furthermore, let $(p_1,p_2,\cdots,p_{N-1})$ be a probability vector corresponding to the IFS $\mathcal{J}_1.$ Then there exist a unique Borel probability measure $\mu^*$ supported on the graph $G({_{x_1} \mathfrak{I}^{\beta}h})$ such that $$\mu^*=\sum_{k \in T}p_k\mu^* \circ H_k^{-1}. $$
	\end{theorem}
	\begin{proof}
		Let us first recall the functional equation satisfied by the FIF $h$:
		$$h\big(P_k(t) \big)= \alpha_k h (t)+ q_k (t), ~~\forall~t \in J,~ k \in T.$$
		Now, for coordinate function $h_i$ of $h,$ we have
		\begin{equation*}
			\begin{aligned}
				_{x_1} &\mathfrak{I}^{\beta}h_i\big(P_k(t) \big)\\& =\frac{1}{\Gamma (\beta)} \int_{x_1} ^{P_k(t)} (P_k(t)-\eta)^{\beta-1} h_i(\eta)~\mathrm{d}\eta\\ &= \frac{1}{\Gamma (\beta)} \int_{x_1} ^{P_k(x_1)} (P_k(t)-\eta)^{\beta-1} h_i(\eta)~\mathrm{d}\eta+ \frac{1}{\Gamma (\beta)} \int_{P_k(x_1)} ^{P_k(t)} (P_k(t)-\eta)^{\beta-1} h_i(\eta)~\mathrm{d}\eta \\ &= \frac{1}{\Gamma (\beta)} \int_{x_1} ^{P_k(x_1)} (P_k(t)-\eta)^{\beta-1} h_i(\eta)~\mathrm{d}\eta+ \frac{a_k}{\Gamma (\beta)} \int_{x_1} ^{t} (P_k(t)-P_k(\eta))^{\beta-1} h_i(P_k(\eta))~\mathrm{d}\eta \\ &= \frac{1}{\Gamma (\beta)} \int_{x_1} ^{P_k(x_1)} (P_k(t)-\eta)^{\beta-1} h_i(\eta)~\mathrm{d}\eta+ \frac{a_k^{\beta} \alpha_k}{\Gamma (\beta)} \int_{x_1} ^{t} (t-\eta)^{\beta-1} h_i(\eta)~\mathrm{d}\eta\\& + \frac{a_k^{\beta} }{\Gamma (\beta)} \int_{x_1} ^{t} (t-\eta)^{\beta-1} q_{k,i}(\eta)~\mathrm{d}\eta \\ &= a_k^{\beta} \alpha_k~ _{x_1} \mathfrak{I}^{\beta}h_i(t)+ \frac{1}{\Gamma (\beta)} \int_{x_1} ^{P_k(x_1)} (P_k(t)-\eta)^{\beta-1} h_i(\eta)~\mathrm{d}\eta \\&+ \frac{a_k^{\beta} }{\Gamma (\beta)} \int_{x_1} ^{t} (t-\eta)^{\beta-1} q_{k,i}(\eta)~\mathrm{d}\eta \\ &= a_k^{\beta} \alpha_k~ _{x_1} \mathfrak{I}^{\beta}h_i(t)+ Q_{k,i}(t),
			\end{aligned}
		\end{equation*}
		where $Q_{k,i}(t)= \frac{1}{\Gamma (\beta)} \int_{x_1} ^{P_k(x_1)} (P_k(t)-\eta)^{\beta-1} h_i(\eta)~\mathrm{d}\eta + \frac{a_k^{\beta} }{\Gamma (\beta)} \int_{x_1} ^{t} (t-\eta)^{\beta-1} q_{k,i}(\eta)~\mathrm{d}\eta.$
		In the third line we use $\eta=P_k(t)$ and $P_k(t)=a_k t+d_k.$ Fourth line follows by the functional equation: $h_i\big(P_k(t) \big)= \alpha_k h_i (t)+ q_{k,i} (t)$ and $P_k(t)-P_k(\eta)=a_k (t-\eta).$ Set $Q_k(t):= (Q_{k,1}(t),Q_{k,2}(t), \dots,Q_{k,M}(t)).$ Then, from the above equation of $_{x_1} \mathfrak{I}^{\beta}h_i\big(P_k(t) \big),$ we may write
		$$_{x_1} \mathfrak{I}^{\beta}h\big(P_k(t) \big) = a_k^{\beta} \alpha_k~ _{x_1} \mathfrak{I}^{\beta}h(t)+ Q_{k}(t).$$
		This shows that the Riemann-Liouville fractional integral $_{x_1} \mathfrak{I}^{\beta}h$ of the FIF $h$ is again an FIF generated by the IFS
		$\mathcal{J}_1=\{J\times \mathbb{R}^M;H_1,H_2,\dots,H_{N-1}\}$ where $H_k(t,\boldsymbol{z})=(P_k(t),F'_k(t,\boldsymbol{z})) $ for $k\in T$ and $P_k(t)$ and $F'_k(t,\boldsymbol{z})$ are defined as follows
		$$P_k(t)=a_{k}t+d_k,$$
		$$F'_k(t,\boldsymbol{z})= a_k^{\beta} \alpha_k \boldsymbol{z}+ Q_k(t).$$

		In order to show the other part of the theorem, we first define a metric $d_0$ on $J\times \mathbb{R}^M$ as follows :
		$$d_0((t_1,\boldsymbol{z_1}),(t_2,\boldsymbol{z_2}))= |t_1-t_2|+\|(\boldsymbol{z_1}- {_{x_1} \mathfrak{I}^{\beta}h(t_1)})-(\boldsymbol{z_2}-{_{x_1} \mathfrak{I}^{\beta}h(t_2)})\|$$ for all $(t_1,\boldsymbol{z_1}),(t_2,\boldsymbol{z_2})\in J\times \mathbb{R}^M.$
		Then, it is easy to show that $\big( J\times \mathbb{R}^M,d_0 \big)$ is a complete metric space.
		Following Theorem \ref{W_K}, the map $H_k: J\times \mathbb{R}^M \to J\times \mathbb{R}^M $ is a contraction map with contraction ratio $s_k=\max\{|a_k|,|a_k^{\beta} \alpha_k|\} $ with respect to metric $d_0$ on   $J\times \mathbb{R}^M.$ Now, on similar lines of the proof of Theorem \ref{th3.9}, one can easily prove the existence of a unique measure supported on the graph $G({_{x_1} \mathfrak{I}^{\beta}h})$ such that $$\mu^*=\sum_{k \in T}p_k\mu^* \circ H_k^{-1}. $$ 
		Thus, the proof is complete.
	\end{proof}
	Next, we give an example, which ensures that it is not generally true that $\mu_*$ is absolutely continuous with respect to $\mu^*$ or $\mu^*$ is absolutely continuous with respect to $\mu_*.$  
	\begin{example}
		Let $f:[0,1]\to \mathbb{R}$ be such that $f(t)=1~~\forall~~ t\in [0,1].$ Let ${_{0} \mathfrak{I}^{\beta}f}$ be the Riemann-Liouville fractional integral of $f$. Then \begin{align*}
			{_{0} \mathfrak{I}^{\beta}f(t)}=\frac{1}{\Gamma (\beta)} \int_o ^t (t-\eta)^{\beta-1} f(\eta)~\mathrm{d}\eta=\frac{t^\beta}{\Gamma{(\beta+1)}}.
		\end{align*} 
		For $\beta=1,~~{_{0} \mathfrak{I}^{1}f(t)}=t$. Let $\mu_1$ and $\mu_2$ be the Borel probability measure supported on the graphs of $f$ and ${_{0} \mathfrak{I}^{1}f}$,  respectively. Therefore, $\mu_1(G(f))=\mu_2(G({_{0} \mathfrak{I}^{1}f}))=1$  and $\mu_1(G({_{0} \mathfrak{I}^{1}f}))=\mu_2(G(f))=0$. From this, it is clear that $\mu_1$ is not absolutely continuous with respect to $\mu_2$ and also $\mu_2$ is not absolutely continuous with respect to $\mu_1$. 
		
	\end{example}
	\par
	In the next theorem, we estimate the Hausdorff dimension and the  box-counting dimension of the graph of the Riemann-Liouville fractional integral of a vector-valued FIF and without any assumption, we also determine an upper bound of the Hausdorff dimension of invariant measures supported on the graph of  ${_{x_1} \mathfrak{I}^{\beta}h}.$
	\begin{theorem}
		Let $\mathcal{I}:=\{J\times \mathbb{R}^M;~~W_i :i\in T\}$ be the IFS as defined earlier. Set $\alpha_{\max}=\{|\alpha_k| : k\in T\}$ and $a_{\min}:= \min\{|a_k|: k \in T \}$. 
		\begin{enumerate}
			
			\item If $q_k \in \mathcal{BV}(J,\mathbb{R}^M)~~\text{for all}~~k\in T$ and $\alpha_{\max}< \frac{1}{(N-1)}$, then for $0<\beta<1,$ $_{x_1}\mathfrak{I}^{\beta}{h} \in \mathcal{C}(J,\mathbb{R}^M)\cap \mathcal{BV}(J,\mathbb{R}^M)$. Moreover,
			$$ \dim_H(\text{G}(_{x_1}\mathfrak{I}^{\beta}{h})) = \dim_B(\text{G}(_{x_1}\mathfrak{I}^{\beta}{h}))=1 .$$
			\item Let $ q_k \in  \mathcal{HC}^{\sigma}(J )$ for each $k \in T ,$ where $ \sigma\in (0,1)$ and $  \frac{\alpha_{\max}}{a_{\min}^\sigma}< 1 $. \begin{itemize}
				\item[(i)] If $0<\beta<1$ and $\beta+\sigma\leq 1$, then $_{x_1}\mathfrak{I}^{\beta}h\in \mathcal{HC}^{\beta+\sigma}(J )$. Furthermore,
				$$ 1 \le  \dim_H(\text{G}(_{x_1}\mathfrak{I}^{\beta}h_i)) \le \overline{\dim}_B(\text{G}(_{x_1}\mathfrak{I}^{\beta}h_i)) \le 2- \beta-\sigma~~~~~\forall~~i\in \{1,2,\cdots,M\}.$$
				\item[(ii)] If $\beta+\sigma>1$, then $_{x_1}\mathfrak{I}^{\beta}h$ is differentiable on $J$ and
				$$ \dim_H(\text{G}(_{x_1}\mathfrak{I}^{\beta}{h})) = \dim_B(\text{G}(_{x_1}\mathfrak{I}^{\beta}{h}))=1 .$$ Moreover, 
				$ \frac{d}{dt}{_{x_1}\mathfrak{I}^{\beta}}h(t)\in \mathcal{HC}^{\beta+\sigma-1}(J)$ and $$ 1 \le  \dim_H\bigg(\text{G}\bigg(\frac{d}{dt}{_{x_1}\mathfrak{I}^{\beta}h_i}\bigg)\bigg) \le \overline{\dim}_B\bigg(\text{G}\bigg(\frac{d}{dt}{ _{x_1}\mathfrak{I}^{\beta}h_i}\bigg)\bigg) \le 3- \beta-\sigma~~~~~\forall~~i\in \{1,2,\cdots,M\}.$$
			\end{itemize}
			\item $$\dim_H(\mu^*)\leq \frac{\sum\limits_{k\in T}p_k\log{p_k}}{\sum\limits_{k\in T}p_k\log{s_k}},$$ where $\mu^*$ is an invariant measure corresponding to the IFS $\mathcal{J}_1$ and support of $\mu^*$ is the graph $\text{G}(_{x_1}\mathfrak{I}^{\beta}{h}).$
		\end{enumerate}
	\end{theorem}
	\begin{proof}
		\begin{enumerate}
			\item  By using Theorem \ref{th3.14} and \cite[Lemma 2.2]{Liang1}, we get our required result.
			\item In the light of Theorem \ref{BBVL3} and \cite[Theorem 2.1]{Liang EST}, one can easily get this result.
			\item One can easily prove this part by following the proof of Theorem \ref{dimM}.
		\end{enumerate}
	\end{proof}
	
	\begin{remark}
		The previous result should be compared with Theorem $4$ of \cite{RSY}, wherein for a linear FIF $h$, which is determined by $$\{L_i(x), F_i(x, y)\}_{i=1}^{N-1}, ~~
		\text{where}~~ L_i(x) = a_i x +b_i
		~~\text{and}~~ F_i(x, y) = d_i y + q_i(x)$$ are such that $\sum\limits_{i=1}^{N-1} |d_i| >1 $ and $\dim_B(G( h) ) = D(\{a_i, d_i\})$, where $D(\{a_i, d_i\})$ is a unique number $ t $ satisfying the equation  $\sum\limits_{i=1}^{N-1} |d_i| a_i^{t-1}=1,$
		it is shown that 
		$$\dim_B(G(_{x_1}\mathfrak{I}^{\alpha}h)) = \dim_B(G( h)) - \alpha,$$ for any $0 < \alpha < D(\{a_i, d_i\}) - 1,$ provided  $\dim_B(G(_{x_1}\mathfrak{I}^{\alpha}q_i))=1$ for any $1\leq i\leq N-1$ and $\dim_B(G(_{x_1}\mathfrak{I}^{\alpha}h))>1.$ Here we  evaluate the Hausdorff and box dimension of the graph of the Riemann-Liouville fractional integral of FIFs by imposing less conditions as compared to \cite[Theorem 4]{RSY}. Let us summarize the above discussion as follows: Ruan \cite{RSY} has calculated the exact value of the box dimension of the graph of fractional integral of a linear FIF, however, we obtain an upper bound of the graph of fractional integral of a fractal function belonging to a more general class of fractal functions than linear FIFs.
	\end{remark}

	\section{STATEMENTS AND DECLARATIONS}
{\bf Data availability}
Data sharing not applicable to this article as no datasets were generated or analysed during the current study	
	
{\bf Funding:}
The authors declare that no funds, grants, or other support were received during the preparation of this manuscript.

{\bf Competing Interests:}
The authors have no relevant financial or non-financial interests to disclose.

{\bf Author Contributions:}
All authors contributed equally in this manuscript. 
	
	\bibliographystyle{amsplain}

\end{document}